\documentclass[a4paper, 10pt]{article}
\usepackage[utf8]{inputenc}
\usepackage[T1]{fontenc}

\usepackage{amsmath}
\usepackage{amsfonts}
\usepackage{amsthm}
\usepackage{amssymb}
\usepackage{mathtools}
\usepackage{microtype}
\usepackage[colorlinks=true]{hyperref}
\usepackage{cite}
\usepackage[capitalize]{cleveref}

\crefname{equation}{equation}{equations}

\usepackage{mathrsfs}
\usepackage{fouriernc}

\usepackage{epigraph}

\newcommand{\SL}{\operatorname{SL}}

\newcommand{\Z}{\mathbb{Z}}
\newcommand{\F}{\mathbb{F}}
\newcommand{\Proj}{\mathcal{P}} 
\newcommand{\trace}{\operatorname{tr}}
\newcommand{\Aut}{\operatorname{Aut}}
\newcommand{\Out}{\operatorname{Out}}
\newcommand{\SAut}{\operatorname{SAut}}

\newcommand{\SAutF}[1]{\SAut(\mathbb{F}_{#1})}

\DeclareMathOperator{\Inn}{Inn}
\DeclareMathOperator{\cspan}{\overline{span}}
\DeclareMathOperator{\ran}{ran}

\newtheorem{theorem}{Theorem}
\newtheorem{lemma}[theorem]{Lemma}
\newtheorem{corollary}[theorem]{Corollary}
\newtheorem{proposition}[theorem]{Proposition}
\newtheorem{remark}[theorem]{Remark}

\usepackage{authblk}
\title{  $\Aut(\F_5)$ has property $(T)$}

\author[1]{Marek Kaluba}
\author[2]{Piotr W. Nowak}
\author[3]{Narutaka Ozawa}
\affil[1]{Adam Mickiewicz University, Pozna\'{n}, Poland}
\affil[1,2]{Institute of Mathematics of the Polish Academy of Sciences, Warsaw, Poland}
\affil[3]{Research Institute for Mathematical Sciences, Kyoto, Japan}
\begin{document}
\maketitle

\begin{abstract}
We give a constructive, computer-assisted proof that $\Aut(\F_5)$, the automorphism group of the free group on $5$ generators, has Kazhdan's property~$(T)$.
\end{abstract}

\epigraph{\textit{But they are useless. They can only give you answers.}}{--- Pablo Picasso, speaking of computers}

Kazhdan's property $(T)$ is a powerful rigidity property of groups with many applications.
Few groups are known to satisfy the property and the main classes of examples are:
lattices in higher rank Lie groups, whose (rigid) algebraic structure allows to deduce various properties, including property $(T)$;
groups acting on complexes whose links satisfy a certain spectral condition, including groups acting on buildings;
and certain hyperbolic groups, such as lattices in $\operatorname{Sp}(n,1)$ and some random hyperbolic groups in the Gromov density model.

The main purpose of this article is to prove property $(T)$ for a new group and give an estimate of its Kazhdan constant.
Let $\SAutF{5}$ denote the special automorphism group of the free group on $5$ generators, with $S$ being its standard generating set consisting of transvections.

\begin{theorem}\label{theorem: main - Aut(F5) has (T)}
The group $\SAutF{5}$ has property $(T)$ with Kazhdan constant
\[\kappa(\SAutF{5}, S) > 0.18.\]
\end{theorem}
As part of the proof we provide explicit elements $\xi_i$ in the group ring of $\SAutF{5}$
and $\lambda > 0$ such that
\[\Delta^2-\lambda\Delta \approx \sum_{i=1}^n \xi_i^*\xi_i,\]
up to a small, controlled error.
The $\xi_i$s are obtained via semidefinite optimization.
Then we show that there exists a mathematically exact solution (for a slightly smaller $\lambda$) in close proximity of the approximate one.

This method of proving property $(T)$ was forseen by the third author in  \cite{Ozawa2016} and used effectively by Netzer and Thom \cite{Netzer2015}, Fujiwara and Kabaya \cite{Fujiwara2017} and the first two authors
\cite{Kaluba2017} to reprove property $(T)$ and give new estimates for Kazhdan constants for the
groups $\SL_n(\mathbb{Z})$, $n=3,4,5$.
We describe in detail the theoretical aspects, the algorithm which is used to produce the solution, as well as the certification procedure that allows to obtain,
out of this approximate solution, a mathematically correct conclusion about the existence of an exact one.
In particular we describe the modifications to the algorithm of \cite{Kaluba2017} which made a search for $\xi_i$s in the context of $\SAutF{5}$ and the certification of the result technically possible.

As an immediate consequence of \cref{theorem: main - Aut(F5) has (T)} we obtain
\begin{corollary}
The groups $\Aut(\mathbb{F}_5)$ and $\Out(\mathbb{F}_5)$ have property $(T)$.
\end{corollary}

Questions whether any of the groups above has property $(T)$ is discussed in many places, e.g. \cite[Question 7]{Bridson2006},
 \cite{Lubotzky2001}, \cite[page 345]{Breuillard2014}, \cite[page 4]{Bogopolski2010}, \cite[page 63]{Ellenberg2014}, to name a few.
Applications to the product replacement algorithm were discussed in \cite{Lubotzky2001}.

The group $\Aut(\mathbb{F}_n)$ is known not to have property $(T)$ for $n=2$  and $n=3$ see \cite{McCool1989} as well as \cite{Grunewald2009,Bogopolski2010}.
It was suspected nevertheless that $\Aut(\mathbb{F}_n)$ might have property $(T)$ for $n$ sufficiently large, with all $n\geqslant 4$ being open.
In the case of $n=4$ our approach -- both the one in \cite{Kaluba2017} and the symmetrized version presented below -- did not give a positive answer, in the sense that we were not able to obtain a sufficiently approximate solution on the ball of radius $2$.
There are three possible reasons for this behavior.
One possibility is that $\Aut(\F_4)$ does not have property $(T)$ (as somehow anticipated by \cite{Bogopolski2010}).
Another is that $\Aut(\F_4)$ has property $(T)$, however no $\xi_i$ as above, supported on the ball of radius $2$ exist.
Search for $\xi_i$s supported on the ball of radius $3$ is already too expensive (in terms of memory and computation time) to be handled by our implementation.
Finally, the third possibility is that again $\Aut(\F_4)$ has property $(T)$ which is wittnessed on the ball of radius $2$, but the spectral gap is so small, that the certification process does not yield a positive answer.

It is an interesting question whether the fact that $\Aut(\F_5)$ has property $(T)$ is sufficient to deduce property $(T)$ for $\Aut(\F_n)$ for all $n\geqslant 5$, similarly as in the case of lattices in higher rank Lie groups.
In \cref{section: extrapolating property (T)} we show that $\Aut(\F_{n+1})$ has a large subgroup with property $(T)$, if $\Aut(\F_n)$ has $(T)$.
However, all attempts to prove property $(T)$ for $\Aut(\F_n)$, $n\geqslant 6$ using property $(T)$ for $\Aut(\F_5)$  seem to break down at the currently open Question 12 in \cite{Bridson2006}.

It was proved in \cite{Gilman1977} that for $n\geqslant 3$  the  group
$\Out(\mathbb{F}_n)$ is residually finite alternating and it is a consequence of \cref{theorem: main - Aut(F5) has (T)} that
the corresponding family of alternating quotients of $\Out(\mathbb{F}_5)$
can be turned into a family of expanders. This was proved earlier in greater generality by Kassabov \cite{Kassabov2007}, however in our case
the generating set is explicit and the same in all of these finite groups, in the sense that it is the image of the generating set of $\Out(\mathbb{F}_5)$.
In particular, this gives an alternative and independent of \cite{Kassabov2007} negative answer to question \cite[Question 10.3.2]{Lubotzky1994}.

\cref{theorem: main - Aut(F5) has (T)} also allows us to give an answer to a question of Popa on the existence of certain crossed product von Neumann algebras with 
property (T), see \cref{remark: Popa's question}.

\medskip
\textsl{Note:}
The techniques developed here became later crucial in \cite{Kaluba2018} while proving property $(T)$ for $\Aut(F_n)$ for all $n\geqslant 6$.
However, the case $n=5$ is not accessible via the argument in \cite{Kaluba2018} and the only existing proof of the case $n=5$ is in the current paper.

\paragraph{Acknowledgements}

MK is partially supported by the National Science Center, Poland grant 2015/19/B/ST1/01458 and 2017/26/D/ST1/00103.

NO is partially supported by JSPS KAKENHI, Grant Number 17K05277 and 15H05739.

This project has received funding from the European Research Council (ERC) under the European Union's Horizon 2020 research
and innovation programme (grant agreement no. 677120-INDEX).

This research was supported in part by PL-Grid Infrastructure, grant ID: propertyt2.

\tableofcontents

\section{Property $(T)$, real algebraic geometry and semi\-definite programming}
\label{section: Background}
Recall that a group $G$ generated by a finite set $S$  has property $(T)$ if there exists $\kappa>0$ such that
$$\sup_{s\in S} \Vert \pi_s v-v\Vert \geqslant \kappa \Vert v\Vert,$$
for every unitary representation $\pi$ of $G$ with no non-zero invariant vectors.
The supremum of all such $\kappa$'s that satisfy the condition above is called the Kazhdan constant of $G$ with respect to the generating set $S$
and is denoted $\kappa(G,S)$.
For an excellent overview of property $(T)$, its many descriptions and applications see \cite{Bekka2008}.

Let $G$ be a discrete group generated by a finite set $S=S^{-1}$.
Given a ring $R$ we consider the associated group ring $R G$, that consists of finitely supported functions $\xi\colon G\to R$.
We will use the notation $\xi=\sum_{g\in G} \xi_g g$, where $\xi_g\in R$ for each $g\in G$, to denote the elements of the group ring $RG$.
The product in $RG$ is then defined by the convolution $(\xi \eta)_g=\sum_{h\in G} \xi_{h}\eta_{h^{-1}g}$.
Recall that the augmentation ideal $IG$ is the kernel of the augmentation map $\omega\colon RG\to R$, $\xi\mapsto \sum_{g\in G} \xi_g$.
The group ring $R G$ is equipped with an involution $^*\colon RG\to RG$, induced by the inversion map on $G$ and defined explicitly as $(\xi^*)_g=\xi_{g^{-1}}$ for any $\xi\in R G$.

The unnormalized Laplacian is an element of the real group ring $\mathbb{R}G$ defined as
\[\Delta = \vert S\vert - \sum_{s\in S} s.\]
In \cite{Ozawa2016} the following characterization of property $(T)$ for discrete groups was proved.
\begin{theorem}
A finitely generated group $G$ has Kazhdan's property $(T)$ if and only if there exist a positive number $\lambda>0$ and a finite family $\{\xi_1, \dots, \xi_n\}$ of elements of the real group ring
$\mathbb{R}G$ such that
\begin{equation}\label{equation: Taka's characterization Detla^2-lambda Delta = sum of xi_i^*xi}
\Delta^2-\lambda\Delta = \sum_{i=1}^n \xi_i^*\xi_i.
\end{equation}
\end{theorem}
The characterization above can be used to prove property $(T)$ for a particular group $G$ by providing an explicit solution of
\cref{equation: Taka's characterization Detla^2-lambda Delta = sum of xi_i^*xi}.
In particular, the fact that the right hand side is a finite sum of (hermitian) squares of finitely supported functions allows us to try to obtain the $\xi_i$'s using semidefinite programming.

Let $E\subseteq G$ be a finite subset and let $B_r(e,S)$ denote the ball (centered at $e$) of radius $r$ in the word-length metric on $G$ induced by the generating set $S$.
Although some of the following considerations are true for arbitrary $E$, we use $E = S\cup S^2=B_2(e,S)$ (or $E=B_3(e,S)$) in practice.
Fix $\mathbf{x}$, an ordered basis of the finite dimensional subspace $\langle E\rangle_\mathbb{R} \subseteq \mathbb{R}G$.
Then \cref{equation: Taka's characterization Detla^2-lambda Delta = sum of xi_i^*xi} has a solution in $\langle E\rangle_\mathbb{R}$ if and only if there exists a semi-positive definite matrix $P$ such that
$$\Delta^2-\lambda \Delta = \mathbf{x}^* P \mathbf{x}^T.$$
Indeed, by positive semidefiniteness $P$ can be written as $P=QQ^T$ and then
$$ \mathbf{x}^* P \mathbf{x}^T = (\mathbf{x}Q)^*(\mathbf{x}Q)^T = \sum \xi_i^*\xi_i,$$
where $\xi_i=\mathbf{x}q_i$ for $q_i$ the $i$-th column of $Q$. In what follows we will drop the distinction and use $\xi_i$ for both a column of a matrix and the corresponding group algebra element without mentioning the basis $\mathbf{x}$.

Let $\mathbb{M}_E$ denote the set of real matrices with columns and rows indexed by $E$.
For $t\in G$ define a matrix $\delta_t \in \mathbb{M}_E$ by setting
\[\big(\delta_t\big)_{x,y}=\left\lbrace \begin{array}{ll}
1 & \text{ if }x^{-1}y=t\\
0 &\text{ otherwise}.
\end{array}\right.\]
Equivalently, this is the element $t\in G$ viewed as an endomorphism of $\langle E \rangle_{\mathbb{R}} $ defined by left-regular representation of $G$ on $\mathbb{R} G$.
For matrices $A,B\in \mathbb{M}_E$ define
\[\langle A,B\rangle = \sum_{x,y\in E} A_{x,y} B_{x,y}=\trace(A^T B),\]
where $\trace (A)=\sum_{x\in E} A_{x,x}$ is the standard trace on $\mathbb{M}_E$.
Then
\[\langle \delta_t,P\rangle =\sum_{x^{-1}y=t} P_{x,y},\]
and for every $t\in E^{-1}E$ the value $(\Delta^2-\lambda \Delta)_t$ at $t$ can be expressed as
\[(\Delta^2-\lambda\Delta)_t= \langle \delta_t, P\rangle.\]
This reduces the problem of existence of a solution to
\cref{equation: Taka's characterization Detla^2-lambda Delta = sum of xi_i^*xi}
to the following semidefinite optimization problem in standard form
(we use $P \succcurlyeq 0$ to denote the constraint of $P$ being positive semidefinite).
\begin{align}
\tag{OP}\label{equation: original optimization problem}
\begin{split}
\text{minimize} \quad & -\lambda\\
\text{subject to} \quad & P\succcurlyeq 0,\quad P\in \mathbb{M}_E,\\
& \langle \delta_t, P\rangle = (\Delta^2 - \lambda\Delta)_t\quad \text{for all } t\in E^{-1}E.
\end{split}
\end{align}

There are \emph{solvers} specialized in solving such problems numerically. Once a solution $(P, \lambda_0)$ is obtained numerically (up to specified precision), the next step is to \emph{certify} its correctness.
This is of utmost importance, as the numerical solution by itself does not provide mathematical certainty that \cref{equation: Taka's characterization Detla^2-lambda Delta = sum of xi_i^*xi} indeed has a solution in $\mathbb{R} G$.
The solution $(P, \lambda_0)$ gives only $\Delta^2 - \lambda_0 \Delta \approx \mathbf{x}P\mathbf{x}^T$, an approximate equality.
E.g. ``positive semidefinite'' matrix $P$ returned by the solver may have negative eigenvalues which are very close to $0$ (i.e. up to the requested precision),
or the linear constraints defined by $\Delta^2 - \lambda \Delta$ might be slightly violated.
Moreover, even though some solvers claim to certify the solution, it is done in floating point arithmetic, which provides no mathematical certainty, see \cite{Neumaier2006}.
Our certification process turns the approximate solution into a proof of the existence of an exact solution, at the cost of decreasing~$\lambda_0$.

The process consists of finding (the real part of) the square root $Q$ of $P$ (i.e. $QQ^T \approx P$), and
projecting the obtained matrix $Q$ onto the augmentation ideal,
simply by subtracting the mean value of the $i$-th column of $Q$ from each entry in that column.
This procedure, e.g. performed in rational (multiprecision) arithmetic, provides an explicit matrix $\overline{Q}$, whose columns correspond to elements of the augmentation ideal supported on $E$.
In our case this is done in interval arithmetic as described in detail later.
This process introduces additional error into each element of $\mathbf{x}^*\overline{Q}\left(\mathbf{x}\overline{Q}\right)^T = \sum_i \xi_i^*\xi_i$, which in general depends on the accuracy obtained by the chosen numerical solver.
However (and most importantly) after the projection,
\cref{lemma: Delta is an order unit for the augmentation ideal} below allows to dominate $r = \Delta^2 - \lambda_0\Delta - \sum_i \xi_i^*\xi_i$, the remainder of the solution.
Note that $r$ gathers both the inaccuracy of the solver and the error introduced by the projection.

For $\xi\in \mathbb{R}G$ let $\Vert \xi\Vert_1=\sum_{g\in G}\vert \xi_g\vert$ be the norm of $\xi$ in $\ell_1(G)$.
We write $a \geqslant b$ for elements $a,b$ of a group ring to denote that $a-b$ enjoys a decomposition into sum of (hermitian) squares.
The following lemma allows to estimate the magnitude of errors involved in our computations and makes certification possible.
For a more thorough treatment of order units see \cite{Ozawa2016}, as well as \cite{Jameson1970,Schmudgen2009} for a more general context.

\begin{lemma}[\cite{Ozawa2016}, \cite{Netzer2015}]\label{lemma: Delta is an order unit for the augmentation ideal}
$\Delta$ is an order unit for the augmentation ideal $IG$, i.e.
for all self-adjoint elements $r\in IG$ there exists $R_0$ such that for all $R\geqslant R_0$ we have
\begin{equation}\label{equation: r+RDelta is PDF}
r+R\Delta\geqslant 0
\end{equation}
that is, $r+R\Delta$ allows a sum of squares decomposition.

Moreover, if $r$ is supported on $B_{2^{m}}$ then
\[R_0 \leqslant 2^{2m-1}\|r\|_1.\]
\textup{(}If $S$ contains no involution, then $R_0\leqslant 2^{2m-2}\|r\|_1$ suffices.\textup{)}
\end{lemma}

This allows to conclude that $r +\varepsilon\Delta \geqslant 0$, for an appropriately chosen $\varepsilon>0$ and thus
\[\Delta^2 - (\lambda_0 - \varepsilon ) \Delta = \sum_i \xi_i^*\xi_i + r + \varepsilon\Delta \geqslant 0.\]
When $\varepsilon$ can be chosen sufficiently small in relation to $\lambda_0$ (so that $\lambda_0-\varepsilon>0$), we can conclude that
\cref{equation: Taka's characterization Detla^2-lambda Delta = sum of xi_i^*xi} indeed has a solution in $\mathbb{R}G$.

\subsubsection*{Complexity of the problem}
The size of the set $E$ translates directly to the computational complexity of optimization problem \eqref{equation: original optimization problem}:
while each element $\xi_i$ is (by its definition) supported on $E$,
\cref{equation: Taka's characterization Detla^2-lambda Delta = sum of xi_i^*xi} defines $|E^{-1}E|$ linear constraints and $|E|^2$ variables in one semidefinite constraint of size $|E|\times |E|$.
It seems to be an interesting problem to understand the influence of the choice of the set $E$ on the obtained bound $\lambda_0$ and numerical properties of the problem.

The optimization problem has been solved numerically in several cases in \cite{Netzer2015, Fujiwara2017, Kaluba2017} yielding new estimates for $\lambda$.
The computations were successful for $E=B_2(e,S)$ e.g. for the groups $\SL_n(\Z)$ with $n=3,4,5$.
Finding a numerical solution of the problem on a computer may not be feasible if the size of the set $E$ is too large, and in fact this is the case for the groups $\SL_n(\Z)$, $n\geqslant 6$, or the groups $\SAutF{n}$ when $n\geqslant 4$.
For instance, the case of $\SAutF{5}$ results in $21\,538\,881$ variables and $11\,154\,301$ constraints, making it a prohibitively large problem.
In order to remedy this and to be able to find a solution to optimization problem \eqref{equation: original optimization problem}
e.g. for $\SAutF{5}$, we will use the symmetries of the set $E$ to reduce the problem's size.

\section{Problem symmetrization}

The size and computational complexity of optimization problem \eqref{equation: original optimization problem} can be significantly decreased by exploiting its rich symmetry derived from the group structure.
Roughly speaking, we will replace solving a large problem by solving many smaller problems and patching the solutions together to obtain a solution to the original, larger problem.
While there are $|E|^2$ variables in the original problem, the number of variables of the symmetrized version is $m_1^2+\dots+m_k^2$, where $m_k$ are the dimensions of the individual component problems.
Note that $\sum_i m_i \leqslant |E|$ i.e. the latter is much smaller than the former.
Moreover using orbit constraints we will reduce the number of constraints significantly.
In the case of $\SAutF{5}$ the result is a problem consisting of $13\,232$ variables in $36$ semidefinite constraints and $7\,229$ linear constraints, which can be realistically attacked with a numerical solver.
However, to be able to perform our computations we will need to give a different parametrisation of the space of possible solutions by the means of orbit reduction and the Wedderburn decomposition.

A somewhat parallel exposition of semidefinite programs size reduction using its symmetry is discussed in detail in \cite{deKlerk2011,Bachoc2012}.
We would like to point out that a numerical approach to numerical symmetrization of optimization problems that does not use
group representation theory is described in \cite{Murota2010}, which may be applicable e.g. to finitely presented groups where the symmetry is not clearly visible.

\subsection{Invariant SDP problems}
Given an automorphism $\sigma$ of $G$ and $\xi\in R G$, denote by $\sigma(\xi)$ the element of $RG$ defined by  $\sigma(\xi)_g=\xi_{\sigma(g)}$.
Let $\Sigma_E$ denote the group
\[\Sigma_E\coloneqq \left\lbrace \sigma \in \Aut(G): \sigma(\Delta)=\Delta \text{ and } \sigma(E) =E\right\rbrace.\]
The group $\Sigma_E$ is finite, determined by its image in $\operatorname{Sym}(S)$, the symmetric group on $S$.
Let $\Sigma$ denote any subgroup of $\Sigma_E$.
The action of $\Sigma$ on $E$ induces an action of $\Sigma$ on $\mathbb{M}_E$ by the formula
\[(\sigma(T))_{x,y} =T_{\sigma^{-1}(x),\sigma^{-1}(y)}.\]
The subspace of matrices invariant under this action will be denoted $\mathbb{M}_E^\Sigma$.

\begin{lemma}\label{lemma: Delta^2-lambda Delta is invariant under Sigma}
The expression $\Delta^2-\lambda\Delta$ is invariant under $\Sigma$.
\end{lemma}
\begin{proof}
The verification is straightforward:
\begin{align*}\label{equation: invariance of LHS under Sigma}
(\Delta^2 -\lambda \Delta)_{\sigma(t)}&= \sum_{g\in G} (\Delta-\lambda I)_g \Delta_{g^{-1} \sigma(t)}\\
&= \sum_{g\in G} (\Delta-\lambda I)_g \Delta_{\sigma\left(\sigma^{-1}(g^{-1})t\right)}\\
&= \sum_{g\in G} (\Delta-\lambda I)_{\sigma^{-1}(g)} \Delta_{\sigma^{-1}(g^{-1})t}\\
&=(\Delta^2 -\lambda \Delta)_t,
\end{align*}
since, in particular, $\sigma^{-1}$ is a bijection.
\end{proof}

\begin{lemma}\label{lemma: properties of sigma}
We have $\delta_{\sigma(t)}=\sigma(\delta_t)$ and $\delta_{t^{-1}}=\delta_t^T$.\qed
\end{lemma}

In \cite{Bachoc2012} a semidefinite problem is said to be invariant with respect to an action of a group $G$ if for every solution $P$, $gP$ is also a solution for every $g\in G$.

\begin{proposition}
Optimization problem \eqref{equation: original optimization problem} is $\Sigma$-invariant.
\end{proposition}
\begin{proof}
Let $P\in \mathbb{M}_E$ be a solution to \eqref{equation: original optimization problem}.
We need to show that $\sigma(P)$ is also a solution to the same problem for every $\sigma\in \Sigma$; i.e.,
\[(\Delta^2-\lambda\Delta)_t= \langle \delta_t, \sigma(P)\rangle,\]
for every $t\in E^{-1} E$. We have
\begin{align*}
\langle \delta_t, \sigma(P)\rangle &= \sum_{x^{-1}y=t} P_{\sigma^{-1}(x),\sigma^{-1}(y)}\\
&=\sum_{\sigma(x')^{-1}\sigma(y')=t} P_{x',y'}\\
&=\sum_{\sigma(x'^{-1}y')=t} P_{x',y'}\\
&=\langle \delta_{\sigma^{-1}(t)}, P\rangle.
\end{align*}
The latter is equal to $(\Delta^2-\lambda\Delta)_{\sigma^{-1}(t)}=(\Delta^2-\lambda\Delta)_t$, by \cref{lemma: Delta^2-lambda Delta is invariant under Sigma}.
\end{proof}

\noindent In particular, convexity yields
\begin{corollary}
Let $P\in \mathbb{M}_E$ be a solution to problem \eqref{equation: original optimization problem} for some $\lambda>0$.
Then there exists $P\in\mathbb{M}_E^{\Sigma}$ that also solves \eqref{equation: original optimization problem} for the same $\lambda>0$.
\end{corollary}

The corollary above shows that we may as well search for an invariant solution.

\subsection{Orbit symmetrization}

Since $\Delta^2-\lambda\Delta$ is $\Sigma$-invariant, it is well defined on the orbit space $E^{-1}E\big/\Sigma$.
While we can decompose $E$ and $E^{-1}E$ into the orbits of $\Sigma$,
it is impossible to naively formulate problem \eqref{equation: original optimization problem} in ``orbit variables'' $E\big/\Sigma$,
as the constraint matrices $\delta_t$ are not well defined in $\mathbb{M}_{E/\Sigma}$.
However, since the solution $P$ is $\Sigma$\nobreakdash-\hspace{0pt}invariant,
constraint matrices can be averaged over orbits.
Let $[t]_{\Sigma} = [t]$ denote the orbit of $t\in E^{-1} E$ under the action of $\Sigma$.
We define
\[\delta_{[t]} = \frac{1}{|\Sigma|}\sum_{\sigma\in\Sigma} \delta_{\sigma(t)},\]
which encodes $\left(\Delta^2-\lambda\Delta\right)_{[t]}$, the value of $\Delta^2-\lambda\Delta$ at (any point of) $[t]$.
The orbit symmetrization of problem \eqref{equation: original optimization problem} can then be written as:
\begin{align*}
\begin{split}
\text{minimize} \quad & -\lambda\\
\text{subject to} \quad & P \succcurlyeq 0, P \in \mathbb{M}_E \\
& \left(\Delta^2 - \lambda\Delta\right)_{[t]} =
\left\langle \delta_{[t]}, P \right\rangle\quad \text{for all}\quad [t]\in E^{-1}E/\Sigma.
\end{split}
\end{align*}
The problem reduces the number of constraints in problem \eqref{equation: original optimization problem}
from $\left|E^{-1}E\right|$ to $\left|E^{-1}E\big/\Sigma\right|$.
However, since $P$ constitutes just one semidefinite constraint of size $|E|\times |E|$,
and each constraint is expressed by $|E|\times |E|$-matrix multiplication,
the problem is still numerically hard to solve efficiently.

\subsection{Block-diagonalization via Wedderburn decomposition}
Recall that the orthogonal dual $\widehat{G}$ of a group $G$ is the family of equivalence classes of irreducible orthogonal representations of $G$.
We will work under the assumption that all irreducible characters of $\Sigma$, the subgroup of the group of symmetries of $E$, are real.
By $n\mathbf{1}_\mathbb{R}$ we denote the $n$-fold direct sum of the trivial, $1$-dimensional real representation.

Let $\varrho_E$ denote the representation of $\Sigma$ on $\ell_2(E)$ induced by the permutation action of $\Sigma$ on $E$.
We can decompose (up to unitary equivalence) $\varrho_E \cong \bigoplus_{\pi \in \widehat{\Sigma}} m_\pi\pi^{}$ into $\pi$-isotypical summands, where each irreducible representation $\pi$ occurs with multiplicity $m_\pi$.
Since the matrix algebra $C^*(\varrho_E)$ (generated by $\varrho_E(\sigma)$ for $\sigma\in \Sigma$) is semisimple, we can find its Wedderburn decomposition: an isomorphism of $C^*(\varrho_E)$ and a direct sum of simple matrix algebras:
\[
C^*(\varrho_E) \to \bigoplus_{\pi\in \widehat{\Sigma}} \mathbb{M}_{\dim\pi} \otimes m_{\pi} \mathbf{1}_{\mathbb{R}} .
\]
As $\Sigma$-invariant matrices in $\mathbb{M}_E$ coincide with $\left(C^*(\varrho_E)\right)'$, the commutant of the algebra $C^*(\varrho_E)$, we obtain
\[\mathbb{M}_E^\Sigma \cong \left(C^*(\varrho_E)\right)' \to \bigoplus_{\pi\in \widehat{\Sigma}} \dim \pi\mathbf{1}_{\mathbb{R}} \otimes \mathbb{M}_{m_{\pi}} .\]

We show how to explicitly realize and block-diagonalize the Wedderburn isomorphism on $\mathbb{M}_E^\Sigma$ by using a \emph{minimal projection system} in $\mathbb{R}\Sigma$.
Denote by $\chi_\pi$ the central projection in $\mathbb{R}\Sigma$ defined by the character of an irreducible representation $\pi$.
Let $\left\{p_\pi\right\}_{\pi\in\widehat{\Sigma}}$ be a minimal projection system;
i.e., a set of self-adjoint, primitive idempotents in $\mathbb{R}\Sigma$ satisfying
$p_\psi p_\pi = 0$ for $\psi \neq \pi$ and
$\chi_\psi p_\pi = \delta_{\psi,\pi}p_\pi$,
where $\delta_{\psi,\pi}$ denotes the Kronecker delta.

It follows from the definition of $p_\pi$ that the image of $\varrho_E(p_\pi)$ is a subspace of $\ell_2(E)$ of dimension $m_\pi$.

Let $U_\pi$ be the matrix realizing the
orthogonal projection from $\ell_2(E)$ onto $\operatorname{im} \rho_E(p_\pi)$,
followed by an isometric isomorphism from
$\operatorname{im} \rho_E(p_\pi)$ onto $\mathbb{R}^{m_\pi}$.
For $A\in \mathbb{M}_E^\Sigma$ we define
\[\Theta_{\pi} (A) = \dim\pi \cdot U_\pi A U_\pi^T.\]
Under the Wedderburn isomorphism above we have
\begin{align*}
A \mapsto & \bigoplus_{\pi\in\widehat{\Sigma}}\ \dfrac{1}{\dim\pi}\cdot \left(\dim\pi\mathbf{1}_{\mathbb{R}}\right) \otimes \Theta_\pi(A),
\intertext{
and thus the map}
A \mapsto &\bigoplus_{\pi \in \widehat{\Sigma}} \Theta_\pi(A)
\intertext{defines a  block-diagonalizing linear isomorphism}
\Theta = \bigoplus_{\pi} \Theta_{\pi} \colon \mathbb{M}_E^\Sigma \to &\bigoplus_{\pi}\mathbb{M}_{m_\pi},
\end{align*}
which satisfies $\trace(A) = \sum_{\pi \in \widehat{\Sigma}} \trace(\Theta_{\pi}(A))$.

\subsection{Symmetrization}
For the purposes of numerical computation, instead of working with $\Theta$ and a basis of $\mathbb{M}_E^\Sigma$
it is advantageous to work with the standard bases of each of $\mathbb{M}_{m_\pi}$ and translate the constraints to the new basis via $\Theta$.

The following optimization problem is a counterpart of problem \eqref{equation: original optimization problem} combining both the orbit symmetrization and the block-diagonalization reduction:
\begin{align}
\tag{SOP}\label{equation: symmetrized optimization problem}
\begin{split}
\text{minimize}\quad & -\lambda\\
\text{subject to} \quad & \left\{P_\pi\succcurlyeq 0, \quad P_\pi\in \mathbb{M}_{m_{\pi}} \right\}_{\pi\in \widehat{\Sigma}}, \\
& \left(\Delta^2-\lambda\Delta\right)_{[t]} = \sum_{\pi \in \widehat{\Sigma}} \left\langle \Theta_{\pi}\left(\delta_{[t]}\right), P_{\pi}\right\rangle \text{ for } [t]\in E^{-1}E\big/\Sigma.
\end{split}
\end{align}
This allows for a reduction of the number of variables in problem \eqref{equation: original optimization problem}
from $|E|^2$ to $\sum_{\pi} m_\pi^2$ (which turns out to be drastically smaller in practice) at the cost of increasing the numerical complexity of the constraints.
Moreover, as most solvers can exploit the block-diagonal structure, the sizes of matrices the solver needs to compute with are much decreased.

\begin{proposition}\label{proposition: reconstruction of a solution}
Let $\left(\lambda_0, \{P_\pi\}_{\pi\in \widehat{\Sigma}}\right)$ be a solution to problem \eqref{equation: symmetrized optimization problem}.
Then $(\lambda_0, P)$ is a solution to problem \eqref{equation: original optimization problem}, where
\[P=\frac{1}{|\Sigma|}\sum_{\sigma\in\Sigma}\sum_{\pi\in \widehat{\Sigma}} \dim{\pi}\cdot\sigma\left(U_{\pi}^T P_{\pi} U_{\pi}\right).\]
\end{proposition}
\begin{proof}
All we need to check is that $\langle \delta_t, P\rangle = \left(\Delta^2 - \lambda_0 \Delta\right)_t$ for all $t\in E^{-1}E$.
We have
\begin{align*}
\langle \delta_t, P\rangle &=
\left\langle
    \delta_t, \dfrac{1}{\vert \Sigma\vert}
    \sum_{\sigma\in \Sigma}
        \sum_{\pi\in \widehat{\Sigma}}
            \dim\pi \cdot  \sigma\left(U_\pi^T P_\pi U_\pi\right)
\right\rangle\\
&=
\dfrac{1}{\vert \Sigma\vert}
\sum_{\sigma\in \Sigma}
    \sum_{\pi\in \widehat{\Sigma}}
        \dim\pi
        \left\langle
            \delta_t, \sigma\left(U_\pi^T P_\pi U_\pi\right)
        \right\rangle\\
&=
\dfrac{1}{\vert \Sigma\vert}
\sum_{\sigma\in \Sigma}
    \sum_{\pi\in \widehat{\Sigma}}
        \dim\pi\cdot
        \trace \left(
            \delta_{\sigma^{-1}(t)}^T U_\pi^T P_\pi U_\pi
        \right)\\
&=
\dfrac{1}{\vert \Sigma\vert}
\sum_{\sigma\in \Sigma}
    \sum_{\pi\in \widehat{\Sigma}}
        \trace\left(
            \dim\pi \cdot U_\pi \delta_{\sigma^{-1}(t)}^T U_\pi^T P_\pi
        \right)\\
&=
\sum_{\pi\in \widehat{\Sigma}}
    \left\langle
        \dim\pi
        \dfrac{1}{\vert \Sigma\vert}
        \sum_{\sigma\in \Sigma}
            U_\pi\delta_{\sigma^{-1}(t)}U_\pi^T, P_\pi
    \right\rangle\\
&= \sum_{\pi\in \widehat{\Sigma}}
    \left\langle
        \dim\pi\cdot
        U_\pi \delta_{[t]} U_\pi^T, P_\pi
    \right\rangle.
\intertext{
Recalling that
$\Theta_\pi\left(\delta_{[t]}\right) = \dim \pi\cdot U_\pi \delta_{[t]} U_\pi^T$
and using the formulation of the costraints of problem \eqref{equation: symmetrized optimization problem} we continue}
&=
\sum_{\pi\in \widehat{\Sigma}}
    \left\langle
        \Theta_{\pi}\left(\delta_{[t]}\right), P_\pi
    \right\rangle
= \left(\Delta^2 - \lambda_0\Delta\right)_{[t]}
= \left(\Delta^2 - \lambda_0\Delta\right)_t.\qedhere
\end{align*}
\end{proof}

\section{The group $\Aut(\F_n)$}

\subsection{Presentation for $\Aut(\F_n)$}
Let $n\in \Z$ be positive and consider the corresponding free group $\F_n$ on $n$ generators.
Denote by $\Aut(\F_n)$ the group of automorphisms of $\F_n$.
The \emph{special automorphims group} $\SAut(\F_n)$ is defined as the preimage of $\{1\}$ under the map
\[\operatorname{det}\colon\Aut(\F_n)\to \lbrace -1,+1\rbrace,\]
obtained by the composition
\[\Aut(\F_n) \to \Aut(\Z^n)=\operatorname{GL}_n(\Z)\to \Z.\]
The first arrow above is induced by the abelianization homomorphism $\F_n\to \Z^n$, and the second is the classical determinant.
The group $\SAut(\F_n)$ is of index 2 in $\Aut(\F_n)$.
Denote by $\operatorname{Inn}(\F_n)\leqslant \Aut(\F_n)$ the subgroup of inner automorphisms of $\F_n$. The \emph{outer automorphism group}
$\Out(\F_n)$ is defined as the quotient $\Out(\F_n)=\Aut(\F_n)/\operatorname{Inn}(\F_n)$.

Let $\left(s_1,\dots, s_n \right)$ be an ordered set of generators of $\F_n$ and consider the following maps of $\F_n$:
\begin{align*}
R^{\pm}_{i,j} (s_k) & =
    \begin{cases}
        s_ks_j^{\pm 1}  & \text{ if } k=i,\\
        s_k             & \text{ otherwise;}
    \end{cases}
\intertext{and}
L^{\pm}_{i,j} (s_k) & =
    \begin{cases}
        s_j^{\pm 1} s_k & \text{ if } k=i,\\
        s_k& \text{ otherwise.}
    \end{cases}
\end{align*}
The set $S = \left\lbrace R^{\pm}_{i,j}, L^{\pm}_{i,j}\text{ for }1\leqslant i,j \leqslant n, i\neq j \right\rbrace$ acts on generating $n$-tuples of $\mathbb{F}_n$
and consists of automorphisms of $\F_n$ which generate the group $\SAut(\F_n)$\cite{Gersten1984,Lubotzky2001}.
The group $\Aut(\F_n)$ is generated by automorphism in $S$ together with automorphisms permutating and inverting generators of $\F_n$.
For a detailed description of automorphisms in $S$ (commonly known as Nielsen transformations or transvections) see \cite[Section 3.2]{Magnus1966}.

\subsection{Implementation of $\Aut(\F_{n})$}

When computing in $\Aut(\F_{n})$ one usually choses to represent elements either as words in a finite presentation,
or actual functions on $\F_n$ transforming free generating sets of $\F_n$ to free generating sets.
The former aproach allows the group operations to be purely mechanical operations on symbols,
but the recognition of the identity element (the word problem) is a major obstruction to effective computation.
The latter approach, requires storage of both the domain (a generating $n$-tuple) and the image of the domain.
It provides an easier solution to the recognition of the identity problem:
two automorphisms are equal if and only if their values on the standard generating set of $\F_n$ agree.
However, to compute those values one needs to ``equalize'' the domains and compare the images of the domain under both automorphims.
The final step requires only to trivial cancellation.

Instead of using either of the approaches we decided to produce our own implementation,
where each element carries both the structure of a word in a finite presentation, as well as functional information.
Let $\Gamma_n$ be the graph with the vertex set consisting of generating $n$-tuples of $\F_n$ and
an edge connecting two such tuples if one can be obtained from the other by the application of an automorphism $s\in S$ (we label the edge by $s$).
We represent elements of $\Aut(\F_n)$ as paths in the graph $\Gamma_n$ which start at the standard $n$-generating tuple $(x_1 , \ldots, x_n )$.
Such path is represented by a word over alphabet $S$ by collecting edge labels in a natural fashion.
Therefore each such path determines an automorphism of $\F_n$ which takes the tuple $(x_1 , \ldots, x_n )$ (the initial vertex of the path) to the $n$-generating tuple of the terminal vertex.

Given two automorphisms $f = f_1 \ldots f_k$, $g = g_1 \ldots g_{l}$ (written as words over alphabet $S$), to decide the equality problem it is enough to compare
\begin{align*}
f((x_1, \ldots, x_n)) &= (x_1 ,\ldots x_n)^{f_1 \ldots f_{k}} \quad \text{and} \\
g((x_1, \ldots, x_n)) &= (x_1 ,\ldots x_n)^{g_1 \ldots g_{l}},
\end{align*}
where letters $f_i$ of $f$ ($g_j$ of $g$) act on the tuple from the right.

Effectively, we represent $\Aut(\F_n)$ as a finitely presented group on $S$ and solve the word problem in an indirect manner.
Note that in this setting it is sufficient to store in each letter only minimal amount of information
(namely: its type ($R$, or $L$), two indicies ($i,j$) and the exponent ($\pm$)),
as storage of neither the reference basis nor the image is not required.

\subsection{The symmetrizing group $\Sigma$}\label{section: Symmetry Group Sigma}
The group generated by the automorphisms permuting and inverting generators of $\F_n$ is isomorphic to the group $\Z/_2\wr S_n$ (the signed permutation group). Consider the action of $\Z/_2\wr S_n$ on $\Aut(\F_n)$ by conjugation.
Clearly the action preserves the set $S$ of all Nielsen transformations and hence the group $\SAutF{n}$. Moreover the action preserves the word-length metric on $\SAutF{n}$ induced by $S$ (and thus $E=B_2(e,S)$), so we can set $\Sigma = \Z/_2\wr S_n$ in our considerations.

\subsubsection*{Minimal projection system for $\Sigma$}

Since $\Sigma \cong {\Z_2}^n \rtimes S_n$ is a semi-direct product of $S_n$ acting in the natural fashion on the $n$-fold direct product of $\Z_2$, we may obtain a minimal projection system for $\Sigma$ from a minimal projection system of $S_n$.
It is a well known fact (see e.g. \cite[Section~8.2]{Serre1977}) that irreducible representations of the semidirect product are formed in the following fashion.
The dual group ${\widehat{\Z_2}}^n$ (which equals the character group in this case) decomposes into $(n+1)$ orbits of the induced action of $S_n$.
Let $g_i \in {\widehat{\Z_2}}^n$ denote the character which evaluates to $-1$ on the non-trivial element in the first $i$ coordinates and to $1$ otherwise.
Then $\left\{g_i\right\}_{i=0}^n$ forms a complete set of orbit representatives of the action,
and the stabiliser of $g_i$ is $H_i = S_i\times S_{n-i}$, embedded naturally into $S_n$ (under the convention $S_0 = S_1 = \{\operatorname{id}\}$). It is now straightforward that every irreducible representation of $\Sigma$ is of the form
\[\theta_{i,\pi} = \operatorname{ind}_{{\Z_2}^n\rtimes H_i}^\Sigma g_i\otimes \pi,\]
where $g_i\in \widehat{\Z_2}^n$ and $\pi\in \widehat{H_i}$ are trivially extended to ${\Z_2}^n\rtimes H_i$.

A minimal projection system for ${\Z_2}^n\rtimes H_i$ can be described as
\[\left\{ \chi_i p_\pi \right\}_{\pi\in \widehat{H_i}},\]
where $\chi_i$ denotes the central projection associated to $g_i$ and
$\left\{p_\pi\right\}_{\pi\in \widehat{H_i}}$ is a minimal projection system for $H_i$.
Note that irreducible representations of $H_i = S_i\times S_{n-i}$ are tensor products of irreducible representations of the factors,
thus a minimal projection system for $H_i$ can be constructed from those of $S_i$ and $S_{n-i}$, by taking all possible products of minimal projections.
Accordingly,
\[\left\{ \left\{\chi_i p_\pi p_\psi\right\}_{\pi\in \widehat{S_i}, \psi \in \widehat{S_{n-i}}}\right\}_{i=0}^n\]
is a minimal projection system for $\Sigma$.

\subsubsection*{Minimal projection system for $S_n$, $n\leqslant 6$}

We present a minimal projection systems for $S_n$, $n\leqslant 6$ which we used in our computations.
It is well known that irreducible representations of $S_n$ correspond to integer partitions of $n$.
For the trivial and the sign representations (partitions $n_1$ and $1_n$, respectively) we pick $\varepsilon = ()$.
For other representations $\pi$ we fix the element $\varepsilon_\pi$ as in the table

\begin{center}
\footnotesize
\begin{tabular}{c l l }
\hline
$n$ & Partition   & $\varepsilon_\pi$\\\hline
$3$ & $2_1 1_1$         & $\frac{1}{2}\big(() + (1,2)\big)$\\\hline
$4$  & $2_1 1_2$         & $\frac{1}{2}\big(() + (1,2)\big)$ \\
  & $3_1 1_1$         & $\frac{1}{2}\big(() - (1,2)\big)$    \\
  & $2_2    $         & $\frac{1}{2}\big(() + (1,2)\big)$    \\\hline

$5$ & $2_1 1_3   $ & $\frac{1}{2}\big(() + \textup{(}1,2\textup{)}\big)$\\
 & $3_1 1_2   $ & $\frac{1}{4}\big(() + (1,4,3,2) + (1,3)(2,4) + (1,2,3,4)\big)$\\
 & $2_2 1_1   $ & $\frac{1}{3}\big(() + (1,3,2) + (1,2,3)\big)$\\
 & $4_1 1_1   $ & $\frac{1}{2}\big(() - \textup{(}1,2\textup{)}\big)$\\
 & $3_1 2_1   $ & $\frac{1}{3}\big(() + (1,3,2) + (1,2,3)\big)$\\
\hline
$6$ & $2_1 1_4      $ & $\frac{1}{2}  \big(() + (1,2)\big)$\\
 & $3_1 1_3      $ & $\frac{1}{4}  \big(() + (1,2) + (1,2)(3,4) + (3,4)\big)$\\
 & $2_2 1_2      $ & $\frac{1}{5}  \big(() + (1,3,5,2,4) + (1,4,2,5,3) + (1,5,4,3,2) + (1,2,3,4,5)\big)$\\
 & $4_1 1_2      $ & $\frac{1}{4}  \big(() - (1,2) + (1,2)(3,4) - (3,4)\big)$\\
 & $3_1 2_1 1_1  $ & $\frac{1}{12} \big(() + (1,2) + (1,2)(3,4) + (3,4) + (3,4,5) + (1,2)(3,4,5) + $\\
 &                 & $(1,2)(3,5) + (3,5) + (1,2)(4,5) + (4,5) + (3,5,4) + (1,2)(3,5,4)\big)$\\
 & $5_1 1_1      $ & $\frac{1}{2}  \big(() - (1,2)\big)$\\
 & $2_3          $ & $\frac{1}{3}  \big(() + (1,3,2) + (1,2,3)\big)$\\
 & $4_1 2_1      $ & $\frac{1}{5}  \big(() + (1,3,5,2,4) + (1,4,2,5,3) + (1,5,4,3,2) + (1,2,3,4,5)\big)$\\
 & $3_2          $ & $\frac{1}{3}  \big(() + (1,3,2) + (1,2,3)\big)$\\
\hline

\end{tabular}
\normalsize
\end{center}
Recall that $\chi_\pi$ is the central projection corresponding to the irreducible (character of the) representation $\pi$.
By inspecting the appropriate table of characters one can check that $\varphi_\pi( \varepsilon_\pi ) = 1$ for the
character $\varphi_\pi$ corresponding to $\pi$.
Even though projections $\varepsilon_\pi$ are not necessarily mutually orthogonal, it follows from the orthogonality of characters that
\[\left\{p_{\pi} = \chi_\pi \varepsilon_\pi\right \}_{\pi\in \widehat{S_n}}\]
constitutes a minimal projection system for $S_n$.

\section{Description of the algorithm}

Given a group $G$ generated by a symmetric set $S$ and a finite subgroup $\Sigma$ of automorphisms of $G$ preserving $S$:
\begin{enumerate}
   \item generate $E=B_2(e,S)$, $E^{-1}E = B_4(e,S)$ and $\Delta$ (stored in \texttt{delta.jld} as the coefficients vector in $E^{-1}E$);
   \item compute the division table $E^{-1}\times E \to E^{-1}E$ (stored in \texttt{pm.jld});
   \item compute the permutation representation $\varrho_E\colon\Sigma \to \mathbb{M}_E$ (stored in \texttt{preps.jld});
   \item compute the Wedderburn decomposition of $\mathbb{M}_E$, i.e. for every $\pi\in \widehat{\Sigma}$ compute $U_\pi$'s (stored in \texttt{U\_pis.jld})
   \item decompose $E^{-1}E$ into orbits of $\Sigma$ (stored in \texttt{orbits.jld});
   \item use $\Delta$, the orbit structure, the division table and $U_\pi$'s to construct the constraints of symmetrized optimization problem \eqref{equation: symmetrized optimization problem};
   \item solve the symmetrized optimization problem to obtain $\big(\lambda_0, \{P_\pi\}_\pi\big)$ (stored in \texttt{lambda.jld}, and \texttt{SDPmatrix.jld});
   \item reconstruct $P$ according to \cref{proposition: reconstruction of a solution} (stored in \texttt{SDPmatrix.jld});
   \item certify the solution $(\lambda_0, P)$ as described in \cref{section: Certification}.
\end{enumerate}

\subsection{Division table on $E$}
To perform quickly the multiplication $\xi_i^*\xi_i$ we cache the division table $M\colon E^{-1}\times E \to E^{-1}E$, as a matrix $M\in \mathbb{M}_E$ such that
$M_{g,h} = g^{-1}h$.
To avoid indexing entries of $M$ by group elements and storing them in $M$ (due to technical reasons)
we fix a non-decreasing (with the word-length) order $\mathbf{x}$ of elements in $E^{-1}E$,
i.e. such that $\mathbf{x}_0 = e$, $\{\mathbf{x}_j \colon j \in \{1, \ldots, |S|\}\}=S$, and $E \subseteq E^{-1}E$ as the first $|E|$-elements.
Then we can store only the (integer) indices of elements in the division table, 
i.e. if $g \in E$ is the $i$-th element of $E^{-1}E$, ($g = \mathbf{x}_i$) and 
$h\in E$ is the $j$-th element of $E^{-1}E$ ($h=\mathbf{x}_j$),
then $M_{i,j} = k$, where $g^{-1}h = \mathbf{x}_k$ is the $k$-th element of $E^{-1}E$.
Thus once the full $M$ has been populated, we no longer need the actual group elements to perform (twisted\footnote{
Note that given two vectors of values on $\mathbf{x}$ which represent elements $x,y$ of $RG$, using only the division table
(i.e. without referring to the elements of $\mathbf{x}$) we can compute $x^*y$, but not $xy$.})
multiplication of elements of $\mathbb{R}G$ supported on $E$.
In particular, given a solution $(\lambda_0, P)$ of problem \eqref{equation: original optimization problem}, division table $M$ and $\Delta$ (as vector of values on $\mathbf{x}$),
one can compute the sum of squares decomposition $\sum_i \xi_i^*\xi_i$ and compare it with $\Delta^2 - \lambda_0 \Delta$
without the need to access group elements directly.

Note that the full division table is also needed for producing the constraint matrices.

\subsection{Symmetrization}

The minimal projection system $\{p_\pi\}_{\pi\in \widehat{\Sigma}}$ for $\Sigma$ is computed in $\mathbb{Q}\Sigma$ as described in \cref{section: Symmetry Group Sigma}.
Then coefficients are converted to floating point numbers and $\varrho_E(p_\pi)$ is evaluated,
where $\varrho_E$ is the permutation representation of $\Sigma$ on $E$.
Matrix representatives for $U_\pi$ are obtained from $\varrho_E(p_\pi)$ using singular value decomposition.

\subsection{Certification}\label{section: Certification}
The certification process starts by converting the solution of problem \eqref{equation: symmetrized optimization problem}
to a solution $(\lambda_0, P)$ of problem \eqref{equation: original optimization problem} using standard floating-point arithmetic.
Then certification of $(\lambda_0, P)$ follows mostly the procedure described in \cite[Section~2]{Kaluba2017}.
However, instead of passing through rational approximation of $Q = \sqrt{P}$, as described there, we use the interval arithmetic directly.
We turn each entry of $Q$ into an interval of radius $\varepsilon \sim 2.2\cdot10^{-16}$ which contains the computed value.
Then we shift those intervals so that the sum of each colum $q_{i}$ contains $0$ (this corresponds to the projection to the augmentation ideal).
Setting $\xi_i=q_i\mathbf{x}$, as previously, we finally compute the residual \[r = \Delta^2-\lambda_0\Delta - \sum\xi_i^*\xi_i.\] Note that in this approach each $\xi_i$ as well as $r$ is an element of the group ring over real intervals. While by using interval arithmetic we loose track of the actual values of $\xi_i$ and $r$, we have a \textit{mathematical guarantee} that there are $\overline{\xi_i}$ with rational coefficients, $\overline{\xi_i}\in \xi_i$ (with $\in$ on the coefficient level), such that the result $\overline{r}$ of a similar computation performed in rational arithmetic (with $\overline{\xi_i}$'s and $\overline{\lambda_0}$)
satisfies \[\|\overline{r}\|_1 \in \|r\|_1 = \left[r_{\textsl{low}}, r_{\textsl{up}}\right].\]

Although less precise than rational arithmetic, the computation of $\|r\|_1$ from $Q$ as described above is much faster,
as it involves multiplication, addition of (machine) floating point numbers, and directed rounding.
Performing similar calculations in rational arithmetic is much slower (if possible in the case of $\SAutF{5}$ at all), as the numerator of a sum of near-zero rationals grows exponentially.

\subsection{Software details}

Our implementation in Julia programming language \cite{Julia2017} depends on the following packages
\begin{itemize}
   \item Nemo \cite{Nemo2017} package for group theoretical computations;
   \item JuMP \cite{JuMP2017} package for formulation of the optimization problem;
   \item SCS \cite{O'Donoghue2016} solver for solving semi-definite problems.
\end{itemize}

\noindent The software used for computations is contained in the following git-repositories:
\begin{itemize}
   \item package \url{https://git.wmi.amu.edu.pl/kalmar/Groups.jl} for computations in wreath products and automorphism groups of free groups;
   \item package \url{https://git.wmi.amu.edu.pl/kalmar/GroupRings.jl} for computations in group rings;
   \item package \url{https://git.wmi.amu.edu.pl/kalmar/PropertyT.jl} for computations of the spectral gap;
   \item repository \url{https://git.wmi.amu.edu.pl/kalmar/GroupsWithPropertyT} contains the specific functions and routines related to the project, as well as general manual on how to replicate the results.
\end{itemize}

\subsection*{Replication details}

The ball of radius $4$ in $\SAutF{5}$ consists of $11\;154\;301$ elements.
Generating $B_4(e, S)$, decomposing it into $7\,229$ orbits of $\Z_2\wr S_5$ action, computing division table on $B_2(e, S)$ and finding $U_\pi$'s for the Wedderburn decomposition took about $17$h on a $12$-core computer and requires a minimum of $64$GB of RAM.
This part has been performed on the PL-Grid cluster.
Once this has been computed, the optimization step can be performed on a desktop computer, as the actual optimization problem consists of $13\,232$ variables in $36$ semi-definite blocks and $7\,229$ constraints.
The optimization phase had been running for over $800$ hours until the acuracy of $10^{-12}$ has been reached.
The reconstruction of $P$ (according to \cref{proposition: reconstruction of a solution})
and its certification take approximately $3.5$h and $1.5$h, respectively
(times reported correspond to a workstation computer with $4$-core CPU).

The pre-computed division table, orbit decomposition, $U_\pi$'s, as well as the solution $P$ used in the proof of \cref{theorem: main - Aut(F5) has (T)} can be obtained from \cite{Kaluba2017a}.

\section{Proof of \cref{theorem: main - Aut(F5) has (T)}}

We are now in position to prove our main theorem.
As indicated above we set $E = B_2(e,S)$ and obtained a solution of optimization problem \eqref{equation: symmetrized optimization problem}.
\begin{proof}[Proof of \cref{theorem: main - Aut(F5) has (T)}]
Let a solution $(\lambda_0 , P)$ of optimization problem \eqref{equation: original optimization problem} be given.
Compute $Q$ the real part of the square root of $P$.
Construct $\overline{Q}$ as described in \cref{section: Certification} and
let $\xi_i$ be the $i$-th column of $\overline{Q}$.
Recall that the residual is given by
\[r=\Delta^2-\lambda_0\Delta -\sum \xi_i^*\xi_i,\]
and that the $\ell_1$-norm $\|r\|_1 = [r_{\textsl{low}}, r_{\textsl{up}} ]$ is an interval.
By \cref{lemma: Delta is an order unit for the augmentation ideal} we have $ r + 2^2 r_\textsl{up} \Delta \geqslant 0 $ hence
\[\Delta^2-\left(\lambda_0-2^2 r_{\textsl{up}} \right)\Delta \geqslant \sum \xi_i^*\xi_i \geqslant 0.\]
This allows to conclude that the spectral gap satisfies
\[\lambda(G,S)\geqslant \lambda_0-2^2 r_{\textsl{up}} .\]
In the case of the provided solution for $\SAut(\mathbb{F}_5)$ we have
\[\lambda_0 = 1.3, \quad \text{and} \quad \Vert r \Vert_1 \subset [8.30\cdot 10^{-6}, 8.41\cdot 10^{-6}],\]
which leads to certified estimate $\lambda > 1.2999
$. Since the generating set consists of $80$ elements this results in
\[\kappa(\SAutF{5}, S) > 0.18027.\qedhere\]
\end{proof}

\section{Extrapolation of property $(T)$} \label{section: extrapolating property (T)}
In case of arithmetic groups, once we know $\SL_n(\Z)$ has
property $(T)$ for some $n$, it is rather easy to deduce from this fact
that $\SL_m(\Z)$ has property $(T)$ for all $m\geqslant n$, because
$\SL_m(\Z)$ is boundedly generated by finitely many conjugates
of $\SL_n(\Z)$ (see \cite{Shalom2006}, particularly around Section 4.III.7).

However, in the case of $\Aut(\F_n)$ a similar approach seems to break down
at the currently open Question 12 in \cite{Bridson2006}.
Namely, it is not known whether a quotient $Q$ of $\Aut(\F_{n+1})$
must be finite provided that $\Aut(\F_n)$ has finite image in $Q$.
If a counterexample exists, property $(T)$ of $\Aut(\F_n)$ cannot, obviously, tell
anything about such an infinite quotient $Q$.
We nevertheless make an effort to extrapolate property $(T)$ of $\Aut(\F_n)$
to a larger group.

For a group $G$, a normal subgroup $H\leqslant G$  and a unitary representation $(\pi,\mathcal{H})$ of $G$,
let $\Proj_H$ denote the orthogonal projection onto the space $\mathcal{H}^H$ of vectors invariant under $\pi(H)$.
The definition of $\kappa(G,S,\pi)$ from \cref{section: Background}
(where $S$ is a subset of $G$) can be reformulated as the supremum of $\kappa\geqslant 0$ that satisfy
\[
\kappa\| v- \Proj_G v \| \leqslant \max_{s\in S} \| v - \pi_s v\|
\]
for all $v \in \mathcal{H}$.
Recall that a group $G$ has property $(T)$ if and only if there is a finite
(necessarily generating, see \cite[Proposition 1.3.2]{Bekka2008}) subset $S$ such that the corresponding
Kazhdan constant $\kappa(G,S,\pi)$ is strictly positive
for every $\pi$.

Let $\alpha\colon G\to\Aut(H)$, $g\mapsto \alpha_g$, be an action and $T\subseteq H$ be
a subset such that
\[
T_G=\left\lbrace\alpha_g(t) \colon g \in G,\,t\in T\right\rbrace
\]
generates $H$.
Put
\[
\ell(h) = \min\left\lbrace k \colon \exists s_1,\ldots,s_k\in T_G\cup T_G^{-1}
 \mbox{ such that }h=s_1\cdots s_k\right\rbrace.
\]
\begin{proposition}
Let $G\ltimes H$ be the semidirect product given by an action $\alpha\colon G\to\Aut(H)$ and $\ell$ be as above.
Assume that $T$ is a finite generating subset of $H$,
$\Inn(H)$ is a subgroup of $\alpha(G)$ and
\[
L=\max \left\lbrace \ell(t^m) \colon t\in T,\, m\in\mathbb{N} \right\rbrace <\infty.
\]
Then, for any unitary representation $(\pi,\mathcal{H})$ of $G\ltimes H$,
one has
\[
\kappa(G\ltimes H,S\cup T,\pi) \geqslant \dfrac{\kappa(G,S,\pi|_G)}{1+4|T|^{1/2}L}.
\]
In particular, if $G$ has property $(T)$ or $(\tau)$, then so does $G\ltimes H$.
\end{proposition}

\begin{proof}
Let $\kappa = \kappa(G,S,\pi|_G)\leqslant 2$ denote the Kazhdan constant of $G$
with respect to representation $(\pi|_G, \mathcal{H})$.
If $\kappa = 0$ there is nothing to prove, so let us assume that $\kappa>0$.
Given $v\in\mathcal{H}$ we can define
\[C=\max_{g\in S\cup T} \| v - \pi_g v\|.\]
If $w = \Proj_G v$ is the orthogonal projection of $v$ onto the subspace of $\pi|_G $-invariant vectors in $\mathcal{H}$ then $\| v - w \|\leqslant\kappa^{-1}C$.
Moreover, for every $g\in G$ and $t\in T$, one has
\[
\left\| w - \pi_{\alpha_g(t)}w \right\| =
\left\| w - \pi_{gtg^{-1}}w \right\| =
\bigg\| w - \pi_t w \bigg\| \leqslant
4\kappa^{-1}C.
\]
Thus, for every $t\in T$
\begin{align*}
\sup_m \big\| w - \pi_{t^m}w \big\| & \leqslant 4\kappa^{-1}LC,\\
\intertext{which implies}
\left\| \left(1 - \Proj_{\langle t\rangle}\right) w \right\| & \leqslant 4\kappa^{-1}LC,
\end{align*}
for the orthogonal projection $\Proj_{\langle t\rangle}$ onto $\mathcal{H}^{\langle t \rangle}$.

Since $w$ is $G$-invariant and $\Inn(H) \leqslant \alpha(G)$,
the linear functional $\tau(a)=\langle aw,w \rangle$ is a trace on the von Neumann algebra
$M$ generated by $\pi(H)$.
Namely, $\tau(a^*a)=\tau(aa^*)$ for any element $a \in M$.
In particular, if orthogonal projections $p$ and $q$ are Murray--von Neumann equivalent in $M$, i.e., 
$p = v^*v$ and $q = vv^*$ for some $v \in M$, then one has $\tau(p) = \tau(q)$.
We write $p \sim q$ to indicate that $p$ and $q$ are Murray--von Neumann equivalent in $M$.
We note that for any orthogonal projections $p,q\in M$ one has
\[
p-p\wedge q \sim p\vee q -q,
\]
where $p\wedge q$ (respectively, \ $p\vee q$) is the orthogonal projection
onto $\ran p\cap \ran q$ (respectively, $\cspan(\ran p \cup \ran q)$).
Here $\ran p$ denotes the range of $p$ and $\cspan(x)$ the closure of the linear subspace spanned by set $x$.

We have
\[
\|(p-p\wedge q)w\|^2 = \tau(p-p\wedge q) = \tau(p\vee q -q) = \|(p\vee q -q)w\|^2,
\]
see \cite[Proposition V.1.6]{Takesaki2002} for a proof of this fact.
In particular,
\[
\left\|\left(1-p\wedge q\right)w\right\|^2 \leqslant
\left\|\left(1-p\right)w\right\|^2+\left\|\left(p\vee q-q\right)w\right\|^2 \leqslant
\left\|\left(1-p\right)w\right\|^2+\left\|\left(1-q\right)w\right\|^2
\]
(the first inequality follows from the triangle inequality and the equality above; the second is the consequence of $p\vee q - q \leqslant 1 - q$). Since $\Proj_{H}=\bigwedge_{t\in T}\Proj_{\langle t\rangle}$, it follows that
\[
\big\|\left(1-\Proj_H\right)w\big\|
\leqslant \left(
    \sum_{t\in T}\left\|\left(1-\Proj_{\langle t\rangle}\right)w\right\|^2
    \right)^{1/2}
\leqslant 4\kappa^{-1}\left|T\right|^{1/2}LC.
\]
Since $\mathcal{H}^H$ is $G$-invariant (by normality of $H$), $\Proj_H w$ is still $G$-invariant.
Therefore,
\[
\big\| v - \Proj_G v \big\| \leqslant \big\| v - \Proj_H w \big\|
 \leqslant \kappa^{-1} \left(1+4|T|^{1/2}L\right)C.
\]
This proves the claim.
\end{proof}

The above proposition applies to the action of
$\Aut(\F_n)$ on $\F_n=\langle x_1,\ldots,x_n\rangle$ for $n\geqslant 2$,
since $x_i^m = x_j^{-1}\cdot x_j^{} x_i^m$ $(j\neq i)$ has $\ell(x_i^m )\leqslant 2$.

\begin{corollary}
If $\Aut(\F_n)$ has property $(T)$, then
the subgroup
\[
\Gamma =\left\lbrace \theta \in \Aut(\F_{n+1}) \colon \theta(\F_n)\subseteq \F_n\right\rbrace
\]
in $\Aut(\F_{n+1})$ has property $(T)$.
\end{corollary}
\begin{proof}
Let $\left\{x_i\right\}_{i=1}^{n+1}$ denote the standard generating set of $\F_{n+1}$.
Any element $\theta$ as above satisfies
$\theta(x_{n+1})=a x_{n+1}^{\pm 1} b$
for some $a,b\in \F_n$.
This means that the above subgroup is isomorphic to
$(\Aut(\F_n)\times\Z/2)\ltimes(\F_n\times \F_n)$,
where $\Aut(\F_n)$ acts on $\F_n\times \F_n$ diagonally and
$\Z/2$ acts on $\F_n\times \F_n$ by the flip.
Since $\Aut(\F_n)\ltimes \F_n$ has property $(T)$ by the above proposition,
$\Aut(\F_n)\ltimes (\F_n\times \F_n)$ has property $(T)$ as well.
\end{proof}

\begin{remark}\normalfont\label{remark: Popa's question}
In \cite[4) p. 324]{Popa2006} Popa asked for an example of an action of a property $(T)$ group $G$ on $L(\mathbb{F}_n)$, the free group factor, so that the crossed product
$L(\mathbb{F}_n)\rtimes G$ has property $(T)$. The examples considered above provide an answer to Popa's question for $n=5$. 
Indeed, by the above arguments, taking $G=\Aut(\F_{5})$ with its natural action on $L(\mathbb{F}_5)$ satisfies all the necessary conditions. 
Another example is given by taking $G=\Out(\F_5)$, in which case $L(\Aut(\F_5))$ is the crossed product $L(\F_5)\rtimes \Out(F_5)$.
\end{remark}

\section{$\SL_6(\Z)$ and $\SL_4(\Z\langle X\rangle)$}

We have also applied the symmetrized algorithm to certain other groups, for which the non-symmetrized approach as in
\cite{Netzer2015, Fujiwara2017, Kaluba2017}, was not succesful.
In the case of two linear groups, $\operatorname{SL}_6(\mathbb{Z})$ and $\operatorname{SL}_4(\mathbb{Z}\langle X\rangle)$ (generated by the set of elementary matrices $E(4)$ and $E(6)$, respectively), we obtained new estimates for the Kazhdan constants.
For the first group we obtained a certified bound
\[\kappa (\operatorname{SL}_6(\mathbb{Z}), E(6)) \geqslant 0.3812;\]
for the second group the certified bound is
\[\kappa(\operatorname{SL}_4(\mathbb{Z}\langle X\rangle), E(4)) \geqslant 0.2327.\]

\section{Final remarks}

\paragraph{Smallest radius for existence of a solution} As mentioned in the introduction, one of the reasons for our approach not producing an answer in the case of $\Aut(\F_4)$ could be that the equation
\eqref{equation: Taka's characterization Detla^2-lambda Delta = sum of xi_i^*xi} does not have a solution where all $\xi_i$'s are supported in ball of radius $2$.
It is in general unclear, for a group with property $(T)$, what is the smallest radius $r$ such that \eqref{equation: Taka's characterization Detla^2-lambda Delta = sum of xi_i^*xi}
has a solution on the ball of radius $r$.
It is equally unclear what is the radius $r$ for which the optimal $\lambda$ is attained (if it exists).

We have not been able to reprove property $(T)$, using the method presented here, for $\SL_3(\Z[X])$ on the ball of radius 2.
Recall that for $\SL_3(\Z[X])$ property $(T)$ was proved by Shalom \cite{Shalom2006} and Vaserstein \cite{Vaserstein2007}, as well as by
Ershov and Jaikin-Zapirain \cite{Ershov2010}.

Note also that the bound given in \cite{Fujiwara2017} on a ball of radius 2 was better for $\SL_4(\Z)$ than for $\SL_3(\Z)$.
All this suggests that to detect
property $(T)$ (or estimate the Kazhdan constant) of $\SL(n,\Z[X_1,...X_k])$
one needs larger ball as $k$ increases and smaller ball
as $n$ increases (although the Kazhdan constant itself
gets smaller).


\bibliographystyle{abbrvurl}
\bibliography{TforAutF5}{}

\end{document}